\documentclass{article}
\usepackage{amssymb}
\usepackage{amsthm}
\usepackage{amssymb}
\usepackage{amsmath}
\usepackage{hyperref}
\hypersetup{
    bookmarks=true,         
    unicode=true,          
    pdftoolbar=true,        
    pdfmenubar=true,        
    pdffitwindow=true,      
    pdftitle={},    
    pdfauthor={Stéphane Gerbi},     
    pdfcreator={TeXShop},   
    pdfproducer={TeXShop}, 
    pdfkeywords={}, 
    colorlinks=true,       
    linkcolor=blue,  
    citecolor=blue,        
    filecolor=magenta,      
    urlcolor=cyan           
}
\if@compatibility
\let\@newpf\proof \let\proof\relax 

\newenvironment{pf*}[1]{\@newpf[#1]}{\qed\endtrivlist}
\fi
\newtheorem{theorem}{Theorem}[section]

\newtheorem{lemma}{Lemma}[section]

\newtheorem{definition}{Definition}[section]

\newtheorem{remark}{Remark}[section]

\numberwithin{equation}{section}
\title{Exponential decay for solutions to semilinear damped wave equation}
\author{ St\'{e}phane Gerbi\thanks{e-mail:Stephane.Gerbi@univ-savoie.fr, \textbf{corresponding
author.}} ~and Belkacem Said-Houari\thanks{Permanent adress: Laboratoire de Math\'{e}matiques Appliqu\'ees, Universit\'{e} Badji Mokhtar,
B.P. 12 Annaba 23000, Alg\'erie, e-mail:saidhouarib@yahoo.fr}\\[0.3cm]
Laboratoire de Math\'ematiques, Universit\'e de Savoie, \\
73376 Le Bourget du Lac, France
}
\date{}
\begin{document}
\maketitle

\begin{abstract}
This paper is concerned with decay estimate of solutions to the semilinear wave equation with strong damping in a bounded domain. Introducing an appropriate
 Lyapunov function, we prove that when the damping is linear, we can find initial data, for which the solution decays exponentially. This result improves an early one in \cite{GS06}.
\end{abstract}
\section{Introduction}
In this paper we are concerned by the following problem 
\begin{equation} \label{ondes}
\left\{
\begin{array}{ll}
u_{tt}-\Delta{u}-\omega\Delta{u}_{t}+\mu{u}_{t}=u \vert u \vert^{p-2}&  x \in \Omega ,\ t>0 \\[0.1cm]
u(x,t) =0, &  x\in \partial \Omega,\ t>0  \\[0.1cm]
u(x,0) =u_{0}(x), \; u_{t}(x,0) = u_{1}(x) & x \in \Omega \quad .
\end{array}
\right.
\end{equation}
in a bounded regular domain $\Omega\subset\mathbb{R}^{n}$. Here $p > 2$ and $\omega,\mu$ are positive constants.
Only one of this constant must be strictly positive (as pointed out by Gazzola and Squassina \cite{GS06}). We will suppose that $\mu > 0$ and $\omega \geq 0$
(see Remark \ref{omega} for the case $\omega = 0$).
The present problem has been studied by Gazzola and Squassina \cite{GS06}.
In their work, the authors proved some results on the well-posedness and investigate the asymptotic behavior 
of solutions of problem (\ref{ondes}). 
In particular, they showed the global existence and the polynomial decay property of solutions provided 
that the initial data are in the potential well, \cite[Theorem 3.8]{GS06}.
The proof in \cite{GS06} is based on a method used in \cite{IS96} and \cite{I96}. In these works, the authors obtain 
the following differential inequality:
$$
\frac{d}{dt}\Bigl[(1+t) E(t) \Bigl] \leq E(t) \quad ,
$$
where $E$ is the energy of the solution. Unfortunately they obtain a decay rate which is not optimal.

The nonlinear wave equations related to (\ref{ondes}) has been investigated by many authors \cite{B77,E04,HZ88,IS96,MS04,O97,Y02,Z90}.

In the absence of the nonlinear source term, it is well known that the presence of one damping term i.e. 
$\omega>0$ or $\mu>0$ ensures global existence and decay of solutions for arbitrary initial data (see \cite{HZ88,K89}).
For $\omega=\mu=0$, the nonlinear term $u\left\vert u\right\vert ^{p-2}$ causes finite-time blow-up of solutions
with negative initial energy (see \cite{B77,KL78}).

The interaction between the damping and the source terms was first considered by Levine  \cite{L74_1,L74_2}. 
He showed that solutions with negative initial energy blows up in finite time.
When $\omega=0$ and the linear term ${u}_{t}$ is replaced by $u_{t}\left\vert u_{t}\right\vert ^{m-2}$, 
Georgiev and Todorova \cite{GT94} extended Levine's result to the case where $m>2$.
In their work, the authors introduced a method different from the one known as the concavity method. 
They determined suitable relations between $m$ and $p$, for which there is global existence or alternatively finite time
blow-up. Precisely, they showed that the solution continues to exist globally ``in time'' if $m \ge p$ and 
blows up in finite time if $p>m$ and the initial energy is sufficiently negative.  
Vitillaro \cite{V99} extended the results in \cite{GT94} to situations where the damping is nonlinear and the solution has positive initial energy.
Similar results have been also established by Todorova \cite{T98,T99}, for different Cauchy problems.

We recall here that the potential well method introduced by Payne and Sattinger \cite{PS75} 
is also useful and widely used in the litterature to investigate the local existence, global existence and 
asymptotic behavior of the solutions to some problems related to problem (\ref{ondes}) 
(see \cite{T99,V99,E04,O97,I96,IS96,Y02}). 
Introducing a strong damping term $\Delta u_{t}$ makes the problem different from the one considered in \cite{GT94}. 
For this reason less results are, at the present time, known for the wave equation with strong damping 
and many problems remain unsolved (see \cite{GS06}).

The purpose of this paper is to obtain a better decay estimate of solutions to the problem (\ref{ondes}).
More precisely we show that we can always find initial data in the stable set for which the solution of problem (\ref{ondes}) 
decays exponentially. 
The key tool in the proof is an idea of Haraux and Zuazua \cite{HZ88} and Zuazua \cite{Z90}, which is based on the construction of a suitable 
Lyapunov function. 
This kind of Lyapunov function, which is a small perturbation of the energy, has been recently used by Benaissa and Messaoudi \cite{BM05}
to study the exponential decay if a weakly damped semilinear wave equations.
\section{Asymptotic stability}

In this section, we introduce and prove our main result.
For this purpose let us introduce the definition of the solution of problem (\ref{ondes})
given by Gazzola and Squassina in \cite{GS06}.
\begin{definition}\label{defsol}
For $T > 0$, we denote 
\begin{equation*}Y_T = 
\left\{\begin{array}{ll}
u \in C^{0}\left( [0,T] ,H_{0}^{1}(\Omega) \right) \cap &C^{1}\left([0,T], L^{2}(\Omega) \right)
\cap C^{2}\left([0,T], H^{-1}(\Omega) \right) \\
~&u_{t}\in L^{2}\left([0,T], L^{2}(\Omega) \right)
\end{array}
\right\}
\end{equation*}
Given $u_0 \in H_{0}^{1}(\Omega)$ and $u_1 \in L^{2}(\Omega)$, a function $u \in Y_T$
is a local solution to (\ref{ondes}), if $u(0)=u_{0},\ u_{t}(0) = u_{1}$ and
\begin{equation*}
\int_{\Omega} u_{tt}\phi dx + \int_{\Omega}\nabla u\nabla \phi dx + 
\omega \int_{\Omega} \nabla u_{t} \nabla \phi dx + \mu \int_{\Omega} u_{t}\phi dx =
\int_{\Omega} \vert u \vert ^{p-2} u \phi dx,
\end{equation*}
for any function $\phi \in H_{0}^{1}(\Omega) $ and a.e. $t \in [ 0,T] \quad .$

\end{definition}
Let us first define the Sobolev critical exponent $\bar{p}$ as:
\begin{equation*}
\bar{p} = \left\{
\begin{array}{l}
\displaystyle \frac{2N}{N-2},\ \mbox{ for } \omega > 0 \mbox{ and } N \geq 3 \\[0.3cm]
\displaystyle \frac{2N-2}{N-2},\emph{\ }\mbox{ for } \omega = 0 \mbox{ and } N \geq 3
\end{array}
\quad \mbox{ and } \bar{p} = \infty ,\mbox{ if }N = 1, 2 \quad .
\right.
\end{equation*}
We first state a local existence theorem whose proof is written by Gazzola and Squassina, \cite[Theorem 3.1]{GS06}.
\begin{theorem}\label{existence}
Assume $2 < p \leq \bar{p}$. Let $u_{0}\in H_{0}^{1}(\Omega)$ and $u_{1} \in L^{2}(\Omega)$.
Then there exist $T>0$ and a unique solution of (\ref{ondes}) over $[0,T]$
in the sense of definition \ref{defsol}.
\end{theorem}
As in the work of Gazzola and Squassina, \cite{GS06}, we define the global solutions and the blow up solutions.
\begin{definition} \label{Tmax}
Let $2<p\leq \bar{p} \,,\, u_{0}\in H_{0}^{1}(\Omega)$ and $u_{1} \in L^{2}(\Omega)$.
We denote $u$ the solution of (\ref{ondes}).
We define:
$$
T_{max} = \sup\Bigl\{ T > 0 \,,\, u = u(t) \; exists \; on \; [0,T]\Bigr\}
$$
Since the solution $u \in Y_T$ (the solution is ``enough regular''), let us recall that if  
$T_{max} < \infty$, then 
$$
 \lim_{\underset {t < T_{max}} {t \rightarrow T_{max}}} \Vert \nabla u \Vert_2  + \Vert u_t \Vert_2 = + \infty \quad .
$$
If $T_{max} < \infty$, we say that the solution of (\ref{ondes}) blows up and that $T_{max}$ is the blow up time.\\
If $T_{max} = \infty$, we say that the solution of (\ref{ondes})  is global.
\end{definition}

In order to study the blow up phenomenon or the global existence of the solution of (\ref{ondes}), we define the following functions:
\begin{eqnarray}
I(u(t)) & =&\Vert \nabla u(t) \Vert_{2}^{2} - \Vert u(t) \Vert_{p}^{p}, \label{Energy_I}\\
J(u(t)) & = &\frac{1}{2} \Vert \nabla u(t) \Vert_{2}^{2}-\frac{1}{p} \Vert u(t) \Vert_{p}^{p}, \label{Energy_J}
\end{eqnarray}
and
\begin{equation}\label{Energy_E}
E(u(t)) = J(u(t)) +\frac{1}{2}\Vert u_{t}(t)\Vert_{2}^{2}
\end{equation}
To have a lighter writing of $I, \, J \mbox{ and } E$, we will write : 
$$I(t) = I(u(t)) \,,\, J(t) = J(u(t)) \mbox{ and } E(t) = E(u(t)) \; $$
Let us remark that multiplying (\ref{ondes}) by $u_{t}$, integrating over $\Omega$ and using integration by parts we obtain:
\begin{equation} \label{derivE}
\frac{dE(t) }{dt}=-\omega \Vert \nabla u_{t} \Vert_{2}^{2}- \mu \Vert u_{t} \Vert _{2}^{2 }\; , \;\forall t\geq 0.
\end{equation}
Thus the function $E$ is decreasing along the trajectories.
As in \cite{PS75}, the potential well depth is defined as:
\begin{equation}\label{potentialwell}
d=\inf_{u\in H_{0}^{1}(\Omega)\backslash \{0\}} \max_{\lambda \geq 0}J(\lambda u) .
\end{equation}
We can now define the so called ``Nehari manifold'' as follows:
\begin{equation*}
\mathcal{N}=\left\{ u\in H_{0}^{1}(\Omega) \backslash \{0\} ; \; I(t) =0\right\} .
\end{equation*}
$\mathcal{N}$ separates the two unbounded sets:
\begin{equation*}
\mathcal{N}^{+}= \left\{u\in H_{0}^{1}(\Omega) ;\; I(t) >0 \right\} \cup \{ 0\} \;
\mbox{ and } \;
\mathcal{N}^{-}=\left\{ u \in H_{0}^{1}(\Omega) ;I(t) < 0 \right\} .
\end{equation*}
The \textit{stable} set $\mathcal{W}$ and \textit{unstable} set $\mathcal{U}$ are defined respectively as:
\begin{equation*}
\mathcal{W=}\left\{ u\in H_{0}^{1}(\Omega) ;J(t) \leq d\right\} \cap \mathcal{N}^{+} \;
\mbox{ and } \;
\mathcal{U=}\left\{ u \in H_{0}^{1}(\Omega);J(t) \leq d \right\} \cap \mathcal{N}^{-}.
\end{equation*}

It is readily seen that the potential depth $d$ is also characterized by
\begin{equation*}
d= \min_{u\in \mathcal{N}} J\left( u\right) .
\end{equation*}
As it was remarked by Gazzola and Squassina in \cite{GS06}, this alternative characterization of $d$ shows that
\begin{equation} \label{altd}
\beta = \mbox{dist}(0,\mathcal{N}) =\min_{u\in \mathcal{N}} \Vert \nabla u\Vert_2 = \sqrt{\frac{ 2 d p }{p-2}}  \,  > 0 \quad .
\end{equation}
In the lemma \ref{lemme1}, we would like to prove the invariance of the set $\mathcal{N}^{+}$: if the initial 
data $u_0$ is in the set $\mathcal{N}^{+}$ and if the initial energy $E(0)$ is not large
(we will precise exactly how large may be the initial energy), then $u(t)$ stays in $\mathcal{N}^{+}$ forever. 

For this purpose, as in \cite{GS06,V99}, we denote by $C_{\ast}$ the best constant in the Poincar\'{e}-Sobolev
embedding $H_{0}^{1}(\Omega) \hookrightarrow L^{p}(\Omega)$ defined by:
\begin{equation}\label{sobolev}
C_{\ast}^{-1} = \inf\left\{\Vert \nabla u \Vert_2 : u \in  H_{0}^{1}(\Omega), \Vert u\Vert_p = 1 \right\} \quad .
\end{equation}
Let us remark (as in \cite{GS06,V99}) that if  $p < \bar{p}$ the embedding is compact and 
the infimum in (\ref{sobolev}) (as well as in (\ref{potentialwell})) is attained. 
In such case (see, e.g. \cite[Section 3]{PS75}), any mountain pass solution of the stationary problem is a minimizer for 
(\ref{sobolev}) and $C_{\ast}$ is 
related to its energy:
\begin{equation} \label{mountainpass}
d = \frac{p-2}{2 p}\; C_{\ast}^{-2 p/(p-2)} \quad .
\end{equation}
\begin{remark} 
It is well know from the potential well theory, \cite{PS75,E03} \label{stable_unstable}, that
for every solution of (\ref{ondes}), given by Theorem \ref{existence}, only one of the following assumption holds:
\begin {enumerate}
\item[i)] if there exists some $t_{0} \geq 0 \mbox{ such that } u(t_{0}) \in \mathcal{W} \mbox{ and } E(t_{0})<d $, then $ \forall t \geq t_{0} \,,\, u(t) \in \mathcal{W}  \mbox{ and }  E(t)<d$.
\item[ii)]  if there exists some $t_{0} \geq 0 \mbox{ such that } u(t_{0}) \in \mathcal{U}  \mbox{ and } E(t_{0})<d $, then $ \forall t \geq t_{0} \,,\, u(t) \in \mathcal{U}  \mbox{ and }  E(t)<d$.
\item[iii)]  $\forall t \geq 0 \,,\, E(t) \geq d$ \quad .
\end{enumerate}
\end{remark}

We can now proceed in the global existence result investigation. For this sake, let us state two lemmas: 
these two results are stated in \cite[Proof of Therorem 3.8]{GS06} but are not detailed.
For a better understanding of the results, we give a short proof of these two results.
\begin{lemma}\label{lemme1} Assume $2< p\leq \bar{p}$. Let
$u_{0}\in \mathcal{N}^{+}$ and $u_{1}\in L^{2}(\Omega) $. Moreover, assume that $E(0) < d$.
Then 
for any $ 0 < T < T_{max}$,
$u(t,.) \in \mathcal{N}^{+}$ for each $t\in [ 0,T).$
\end{lemma}
\begin{remark}   Let us remark, that if there exists $t_0 \in [0,T)$ such that 
$$E(t_0) < d$$
the  same result stays true. It is the reason why we choose $t_0 = 0$.

Moreover , one can easily see that, from (\ref{mountainpass}), the condition $E(0) < d$ is equivalent to the inequality:
\begin{equation} \label{initial}
C_{\ast}^{p}\left(\frac{2p}{p-2} E(0) \right)^{\frac{p-2}{2}} < 1
\end{equation}
This last inequality will be used in the remaining proofs.
\end{remark}
\begin{proof}
Since $I(u_{0}) >0$,  then by continuity, there exists $T_{\ast}\leq T$ such that 
$I(u(t,.)) \geq 0,$ for all $t\in [0,T_{\ast })$. Since we have the relation:
\begin{equation*}
J(t)   =  \frac{p-2}{2p} \Vert \nabla u\Vert_{2}^{2}+ \frac{1}{p} I(t) \nonumber\\
\end{equation*}
we easily obtain :
\begin{equation*}
J(t) \geq \frac{p-2}{2p} \Vert \nabla u\Vert_{2}^{2}, \quad \forall t\in [0,T_{\ast }) \quad . 
\end{equation*}
Hence we have:
\begin{equation*}
\Vert \nabla u \Vert_{2}^{2} \leq  \frac{2p}{p-2}J(t) \quad .
\end{equation*}
From (\ref{Energy_J}) and (\ref{Energy_E}), we obvioulsy have $\forall t \in [0,T_{\ast}), J(t) \leq E(t)$. Thus we obtain:
\begin{equation*}
\Vert \nabla u \Vert_{2}^{2} \leq  \frac{2p}{p-2}E(t) 
\end{equation*}
Since $E$ is a decreasing function of $t$, we finally have:
\begin{equation} \label{ineqE0}
\Vert \nabla u \Vert_{2}^{2} \leq \frac{2p}{p-2}E(0) ,\,\forall t\in [ 0,T_{\ast }) \;. 
\end{equation}
By definition of $C_{\ast}$, we have:
\begin{equation*}
\Vert u \Vert_{p}^{p} \leq  C_{\ast }^{p} \Vert \nabla u \Vert_{2}^{p} = 
                              C_{\ast }^{p} \Vert \nabla u \Vert_{2}^{p-2}\Vert \nabla u \Vert_{2}^{2}
\end{equation*}
Using the inequality (\ref{ineqE0}), we deduce:
\begin{equation*}
\Vert u \Vert_{p}^{p} \leq C_{\ast }^{p} \left(\frac{2p}{p-2}E(0)\right)^{\frac{p-2}{2}} \Vert \nabla u\Vert_{2}^{2} \quad .
\end{equation*}
Now exploiting the inequality on the initial condition (\ref{initial}) we obtain:
\begin{equation} \label{upp_gradu}
\Vert u \Vert_{p}^{p} < \Vert \nabla u\Vert_{2}^{2} \quad .
\end{equation}
Hence $\Vert \nabla u \Vert_{2}^{2} - \Vert u \Vert_{p}^{p}>0, \; \forall t  \in [ 0,T_{\ast })$, this shows that 
$u(t,.) \in \mathcal{N}^{+}, \; \forall t\in [ 0,T_{\ast})$. 
By repeating this procedure, $T_{\ast }$ is extended to $T$.
\end{proof}
\begin{lemma}\label{lemme2}Assume $2< p\leq \bar{p}$. Let $u_{0}\in \mathcal{N}^{+}$ and $u_{1}\in L^{2}(\Omega) $. 
Moreover, assume that $E(0) < d$. 
Then the solution of the problem (\ref{ondes}) is global in time.
\end{lemma}
\begin{proof}
Since the map $t \mapsto E(t)$ is a decreasing function of time $t$, we have:
$$
E(0)  \geq E(t) =\frac{1}{2}\Vert u_{t} \Vert_{2}^{2}+ \frac{(p-2) }{2p}\Vert \nabla u \Vert_{2}^{2}+\frac{1}{p}I(t) \quad ,
$$
which gives us:
$$ 
E(0) \geq \frac{1}{2}\Vert u_{t}\Vert_{2}^{2}+ \frac{(p-2) }{2p} \Vert \nabla u \Vert_{2}^{2} \quad .
$$
Thus, $\forall t \in [0,T)\,,\, \Vert \nabla u \Vert_2  + \Vert u_t \Vert_2$ is uniformely bounded by a constant depending
only on $E(0)$ and $p$. Then by definition \ref{Tmax}, the solution is global, so $T_{max} = \infty$.
\end{proof}
We can now state the asymptotic behavior of the solution of (\ref{ondes}).
\begin{theorem} \label{exponential} Assume $2< p\leq \bar{p}$. Let
$u_{0}\in \mathcal{N}^{+}$ and $u_{1}\in L^{2}(\Omega) $. Moreover, assume that $E(0) < d$.  
Then there exist two positive constants $\widehat{C}$ and $\xi $ independent of  $t$ such that:
\begin{eqnarray*}
0 < E(t) \leq \widehat{C}e^{-\xi t},\ \forall \, t\geq 0.
\end{eqnarray*}
\end{theorem}
\begin{remark} 
Let us remark that these inequalities imply that there exist positive constants $K$ and $\zeta$ independent of  $t$ such that:
\begin{eqnarray*}
\Vert \nabla u(t)\Vert_2^2 + \Vert u_t(t) \Vert_2^2 \leq K e^{-\zeta t},\ \forall \, t\geq 0.
\end{eqnarray*}
Thus, this result improves the decay rate of Gazzola and Squassina \cite[Theorem 3.8]{GS06}, in which the authors 
showed only the polynomial decay. Here we show that we can always find initial data satisfying $u_{0}\in \mathcal{N}^{+}$
and $u_{1}\in L^{2}(\Omega) $ which verify the inequality (\ref{initial}), such that the solution can decay
faster than $1/t$, in fact with an exponential rate. 
Also, the same situation happens in absence of strong damping ($\omega=0$).
\end{remark}
\begin{proof}
Since we have proved that $\forall t \geq 0 \,,\, u(t)\in \mathcal{N}^{+}$, we already have:
\begin{eqnarray*}
0 < E(t) \quad \forall \, t\geq 0.
\end{eqnarray*}
The proof of the other inequality relies on the construction of a Lyapunov functional by performing a suitable 
modification of  the energy. To this end, for $\varepsilon >0$, to be chosen later, we define
\begin{equation}\label{energy_L}
L(t) = E(t) +\varepsilon \int_{\Omega}u_{t} u dx + \frac{\varepsilon \omega }{2} \Vert \nabla u\Vert_{2}^{2} \quad .
\end{equation}%
It is straightforward to see that $L(t) $ and $E(t) $ are equivalent in the sense that there exist two positive constants
$\beta_{1}$ and $\beta_{2}>0$ depending on $\varepsilon$ such that for $t\geq 0$%
\begin{equation} \label{equivLE}
\beta_{1}E(t) \leq L(t) \leq \beta_{2}E(t) .
\end{equation}
By taking the time derivative of the function $L$ defined above in equation (\ref{energy_L}), using problem (\ref{ondes}),
and performing several integration by parts, we get:
\begin{eqnarray}
\frac{dL(t)} {dt} & = &-\omega \Vert \nabla u_{t}\Vert_{2}^{2} - \mu \Vert u_{t}\Vert_{2}^{2}
+\varepsilon \Vert u_{t} \Vert_{2}^{2}-\varepsilon \Vert \nabla u \Vert_{2}^{2}
\nonumber \\
&& + \varepsilon \Vert u \Vert_{p}^{p} - \varepsilon \mu \int_{\Omega}u_{t} \, u dx \label{dLdt} \quad .
\end{eqnarray}
Now, we estimate the last term in the right hand side of (\ref{dLdt}) as follows.\\
By using Young's inequality, we obtain, for any $\delta>0$
\begin{equation}\label{int_(u_tu)}
\int_{\Omega } u_{t} u dx \leq \frac{1}{4\delta } \Vert u_{t} \Vert_{2}^{2}+ \delta \Vert u \Vert _{2}^{2} \quad .
\end{equation}
Consequently, inserting (\ref{int_(u_tu)}) into (\ref{dLdt}) and using inequality (\ref{upp_gradu}), we have:
\begin{eqnarray}
\frac{dL(t)}{dt} & \leq & -\omega \Vert \nabla u_{t}\Vert_{2}^{2}+ 
\left( \varepsilon \left(\frac{\mu}{4\delta } +1 \right ) - \mu \right) \Vert u_{t} \Vert_{2}^{2}  \nonumber \\
&& + \varepsilon \left( \mu C_{\ast}^{2} \delta + 
\underset{<0} {\underbrace{C_{\ast}^{p}\left( \frac{2p}{(p-2) }E(0) \right)^{\frac{p-2}{2}}-1}}\right) 
\Vert \nabla u \Vert_{2}^{2} \quad . \label{estimdLdt1}
\end{eqnarray}
From (\ref{initial}), we have 
$$
\displaystyle C_{\ast }^{p}\left( \frac{2p} {\left(p-2\right)}E(0) \right)^{\frac{p-2}{2}}-1 < 0 \quad .
$$
Now, let us choose $\delta$ small enough such that:
$$
\displaystyle  \mu C_{\ast}^{2} \delta + 
C_{\ast }^{p}\left( \frac{2p} {\left(p-2\right)}E(0) \right)^{\frac{p-2}{2}}-1 < 0 \quad .
$$
From (\ref{estimdLdt1}), we may find $\eta > 0 $, which depends only on $\delta$, such that:
\begin{equation*}
\frac{dL(t)}{dt} \leq -\omega \Vert \nabla u_{t} \Vert_{2}^{2}+
\left( \varepsilon \left( \frac{\mu}{4 \delta} + 1\right) - \mu \right) \Vert u_{t}\Vert_{2}^{2}
-\varepsilon \eta \Vert \nabla u \Vert_{2}^{2} 
\end{equation*}
Consequently, using the definition of the energy (\ref{Energy_E}), for any positive constant $M$, we obtain:
\begin{eqnarray}
\frac{dL(t)}{dt} & \leq & - M \varepsilon E(t) +
\left(\varepsilon \left( \frac{\mu}{4 \delta} + 1 + \frac{M}{2}\right)  -\mu \right) \Vert u_{t} \Vert_{2}^{2} - 
\omega \Vert \nabla u_{t}\Vert_{2}^{2}  \nonumber \\
&& + \varepsilon \left( \frac{M}{2}- \eta \right) \Vert \nabla u \Vert_{2}^{2} \label{estimdLdt2} \quad .
\end{eqnarray}
Now, choosing $M \leq 2\eta$, and $\varepsilon $ small enough such that
$$
\left(\varepsilon \left( \frac{\mu}{4 \delta} + 1 + \frac{M}{2}\right)  -\mu \right) < 0 \quad,
$$
inequality (\ref{estimdLdt2}) becomes: 
\begin{equation*}
\frac{dL(t) }{dt}\leq -M\varepsilon E(t) ,\ \forall t\geq 0.
\end{equation*}%
On the other hand, by virtue of (\ref{equivLE}), setting $\xi =-M\varepsilon /\beta_{2}$, the last inequality becomes:
\begin{equation}\label{diffineq}
\frac{dL(t)}{dt} \leq -\xi L(t) \;,\quad \forall t\geq0 \quad .
\end{equation}
Integrating the previous differential inequality (\ref{diffineq}) between $0$ and $t$ gives the following estimate for the
function $L$:
\begin{equation*}
L(t) \leq Ce^{-\xi t} \;,\quad \forall t\geq0 \quad .
\end{equation*}%
Consequently, by using  (\ref{equivLE}) once again, we conclude 
\begin{equation*}
E(t) \leq \widehat{C} e^{-\xi t} \;,\quad \forall t\geq0 \quad .
\end{equation*}
This completes the proof.
\end{proof}
\begin{remark}\label{omega}
Note that we can obtain the same results as in  Theorem \ref{exponential} in the case $\omega =0$, 
by taking the following Lyapunov function
\begin{equation*}
L(t) =E(t) + \varepsilon \int_{\Omega }u_{t}udx.
\end{equation*} 
\end{remark}
\begin{remark} It is clear that the following problem:
\begin{equation*}
\left\{
\begin{array}{ll}
u_{tt} - \mbox{\rm div} \left( \displaystyle \frac{\nabla u}{\sqrt{1 + \vert \nabla u \vert^{2}}} \right)
- \omega\Delta{u}_{t}+\mu{u}_{t}=u \vert u \vert^{p-2}&  x \in \Omega ,\ t>0 \\[0.1cm]
u(x,t) =0, &  x\in \partial \Omega,\ t>0  \\[0.1cm]
u(x,0) =u_{0}(x), \; u_{t}(x,0) = u_{1}(x) & x \in \Omega \quad .
\end{array}
\right.
\end{equation*}
could be treated with the same method and we obtained also an exponential decay of the solution if the initial condition
is in the positive Nehari space and its energy is lower that the potential well depth.
\end{remark}
\section*{Acknowledgments} 
The second author was supported by MIRA 2007 project of the R\'{e}gion Rh\^{o}ne-Alpes. This author wishes to thank Univ. de Savoie of Chamb\'{e}ry
for its kind hospitality.
Moreover, the two authors wish to thank the referee for his useful remarks and his careful reading of the proofs presented in this paper.
\bibliographystyle{plain}

\end{document}